\documentclass[final]{amsart} 

\usepackage[boxruled,boxed,lined]{algorithm2e}
\usepackage{amssymb}
\usepackage{amsfonts}
\usepackage{amsmath}
\usepackage{amscd}
\usepackage{amsthm}
\usepackage{graphicx}
\usepackage[dvips]{color}

\newtheorem{theorem}{Theorem}
\theoremstyle{definition}

\theoremstyle{remark}
\newtheorem{remark}{Remark}
\newtheorem{example}{Example}

\newcommand{\draft}{false}

\newcommand{\R}{\mathbb{R}}
\newcommand{\IR}{\mathbb{IR}}

\newcommand{\va}{{\bf a}}
\newcommand{\vb}{{\bf b}}
\newcommand{\vf}{{\bf f}}
\newcommand{\vg}{{\bf g}}

\newcommand{\vx}{{\bf x}}
\newcommand{\vy}{{\bf y}}
\newcommand{\vz}{{\bf z}}

\newcommand{\vnull}{{\bf 0}}

\newcommand{\imid}{\mathrm{mid}}
\newcommand{\iwid}{\mathrm{wid}}

\title{Sensitivity Analysis Using a Fixed Point Interval Iteration}

\author{Alexandre Goldsztejn$^*$}

\thanks{$^*$CNRS -- University of Nantes, France ({\tt Alexandre.Goldsztejn@univ-nantes.fr}).}

\date{}

\begin{document}

\begin{abstract}
	Proving the existence of a solution to a system of real equations is a central issue in numerical analysis. In many situations, the system of equations depend on parameters which are not exactly known. It is then natural to aim proving the existence of a solution for all values of these parameters in some given domains. This is the aim of the parametrization of existence tests. A new parametric existence test based on the Hansen-Sengupta operator is presented and compared to a similar one based on the Krawczyk operator. It is used as a basis of a fixed point iteration dedicated to rigorous sensibility analysis of parametric systems of equations.
\end{abstract}

\maketitle

\section*{Notations}

Vectors are denoted by boldface symbols, and interval, interval vectors and matrices by bracketed symbols. Let $\vf:\mathbb E\longrightarrow \mathbb F$ be a function, and $\mathbb X\subseteq\mathrm{dom}(\vf)$. Then $\vf(\mathbb X):=\{\vf(\vx)\in\mathbb F|\vx\in\mathbb X\}$ is the range of $\vf$ on $\mathbb X$.


\section{The Hansen-Sengupta Existence Test}

The presentation given here follows the one given by Neumaier in \cite{Neum90}. The interval Gauss-Seidel is defined as follows: First in dimension one, 
\begin{equation}
	[\gamma]\bigl([a],[b],[x]\bigr):=\Box \{ x\in[x] \ | \ \exists a\in[a],\exists b\in[b],ax=b \}.
\end{equation}
In the case where $0\notin[a]$, one obtains the expression $[\gamma]\bigl([a],[b],[x]\bigr)=([b]/[a])\cap[x]$ (cf. \cite{Neum90} for the expression in the case $0\in[a]$). Then, the multidimensional Gauss-Seidel is then defined as follows: $[\Gamma]\bigl([A],[\vb],[\vx],[\vz]\bigr):=[\vx']$ where
\begin{equation}
	[x'_i]:=[\gamma]\Bigl( \ [a_{ii}] \ , \ [b_i]-\sum_{j< i}[A_{ij}][x'_j]-\sum_{j> i}[A_{ij}][x_j] \ , \ [z_i] \ \Bigr).
\end{equation}

\begin{remark}
	In the traditional definition of the interval Gauss-Seidel operator, the interval vector $[\vz]$ is equal to $[\vx]$ (and hence does not appear explicitly in its definition). Using instead $[\vz]=\R^n$ disactivates the intersection with the previous domain and can be useful for some applications (cf. Section \ref{s:application}).
\end{remark}

Then, the Hansen-Sengupta operator \cite{Hansen1981} can be defined as follows
\begin{equation}\label{eq:Hansen-Sengupta}
	\tilde \vx+[\Gamma]\Bigl( [X] , -[\vy] , [\vx]-\tilde\vx, [\vz]-\tilde\vx \Bigr).
\end{equation}
where $[X]\in\IR^{n\times n}$ and $[\vx],[\vy],[\vz]\in\IR^n$. The following theorem shows how the Hansen-Sengupta operator can be used to improve the enclosure and prove the existence of solutions (cf. \cite{Neum90}).
\begin{theorem}\label{thm:HS}
	Let $[\vx],[\vy],[\vz]\in\IR^n$, $\tilde\vx\in[\vx]$ and $[X]\in\IR^{n\times n}$ such that: $[\vx]\subseteq[\vz]$, $\vf(\tilde\vx)\in[\vy]$ and $[X]\supseteq\{\frac{d\vf}{d\vx}(\vx)\in\R^{n\times n}|\vx\in[\vx]\}$. If $[\vx']$ denotes \eqref{eq:Hansen-Sengupta} then:
	\begin{enumerate}
		\item $\vx\in[\vx]$ and $\vf(\vx)=\vnull$ implies $\vx\in[\vx']$.
		\item If $\emptyset\neq[\vx']\subseteq \mathrm{int}[\vx]$ then $\vf$ has an unique zero in $[\vx']$.
	\end{enumerate}
\end{theorem}

\begin{remark}
	The interval evaluation of the derivatives can be replaced by Lipschitz interval matrices to release the differentiability hypothesis, and to slope matrices to improve the enclosure (though uniqueness of the solution is lost when slopes are used, cf. \cite{Neum90} for details).
\end{remark}

A preconditioning is usually coupled to this kind of operator: The preconditioned system $C\cdot\vf(\vx)=\vnull$, where $C\in\R^{n\times n}$ is nonsingular, is equivalently solved instead of $\vf(\vx)=\vnull$. The preconditioning matrix $C$ is chosen so that $C\cdot\vf$ is close to the identity in the domain considered, hence improving the efficiency of the operator.


\section{Parametrization of the Hansen-Sengupta Existence Test}\label{s:pp-HS}

Functions with parameters are considered in this section. Let $\vf:\R^p\times\R^n\longrightarrow \R^n$ be a function of $n$ variables and $p$ parameters. Parameters will be denoted by the vector $\va$ and variables by the vector $\vx$. The parametric Hansen-Sengupta operator is expressed applying its non-parametric version to different inputs. A more general parametric Hansen-Sengupta (which was dedicated to quantified parameters thanks to the usage of the Kaucher arithmetic) was proposed and used in \cite{Goldsztejn-SAC2006}.

\begin{theorem}\label{thm:natural-parametrization}
	Let $[\vx],[\vy],[\vz]\in\IR^n$, $[\va]\in\IR^p$, $\tilde\vx\in[\vx]$ and $[X]\in\IR^{n\times n}$ such that: $[\vx]\subseteq[\vz]$, $\vf([\va],\tilde\vx)\subseteq[\vy]$ and $[X]\supseteq\{\frac{d\vf}{d\vx}(\va,\vx)\in\R^{n\times n}|\va\in[\va],\vx\in[\vx]\}$. If $[\vx']$ denotes \eqref{eq:Hansen-Sengupta} then:
	\begin{enumerate}
		\item $\va\in[\va]$ and $\vx\in[\vx]$ and $\vf(\va,\vx)=\vnull$ implies $\vx\in[\vx']$.
		\item If $\emptyset\neq[\vx']\subseteq \mathrm{int}[\vx]$ then for every $\va\in[\va]$, $\vf(\va,\cdot)$ has an unique zero in $[\vx']$.
	\end{enumerate}
\end{theorem}
\begin{proof}
	Fix an arbitrary $\hat\va\in[\va]$ and define $\vg:\R^n\longrightarrow\R^n$ by $\vg(\vx)=\vf(\hat\va,\vx)$. We are going to apply Theorem \ref{thm:HS} to $\vg$. First, $\vg(\tilde\vx)=\vf(\hat\va,\vx) \in [\vy]$. Second, as $\frac{d\vg}{d\vx}(\vx)=\frac{d\vf}{d\vx}(\hat\va,\vx)$,
	\begin{equation}
		\left\{\frac{d\vg}{d\vx}(\vx)\in\R^{n\times n} \ | \ \vx\in[\vx]\right\} \ = \ \left\{\frac{d\vf}{d\vx}(\hat\va,\vx)\in\R^{n\times n} \ | \ \vx\in[\vx]\right\} \ \subseteq \ [X].
	\end{equation}
	Therefore, Theorem \ref{thm:HS} can be applied to $\vg$ and the domain $[\vx]$, and shows that if $[\vx']$ denotes \eqref{eq:Hansen-Sengupta} then
	\begin{enumerate}
		\item $\vg(\vx)=\vnull$ (that is $\vf(\hat\va,\vx)=\vnull$) implies $\vx\in[\vx']$.
		\item $\emptyset\neq[\vx']\subseteq \mathrm{int}[\vx]$ implies the existence of an unique zero of $\vg$ (that is of $\vf(\hat\va,\cdot)$) in $[\vx']$.
	\end{enumerate}
	This holds for every $\hat\va\in[\va]$ and hence concludes the proof.
\end{proof}

An interval extension of $\vf$ can be used to compute $[\vy]$ satisfying $\vf([\va],\tilde\vx)\subseteq[\vy]$. Using the mean-value extension to compute $[\vy]$ and the usual inverse midpoint preconditioning gives rise the following parametric Hansen-Sengupta operator, denoted by $[H]_{\vf,[\va]}([\vx],[\vz])$ in the sequel:
\begin{equation}\label{eq:pp-HS}
	\tilde\vx+[\Gamma]\Bigl( \ C\cdot[X] \ , \ -C\cdot[\vf](\tilde\va,\tilde\vx) - \bigl(C\cdot[A]\bigr)\cdot([\va]-\tilde\va) \ , \ [\vx]-\tilde\vx \ \Bigr),
\end{equation}
with $[X]=[\frac{d\vf}{d\vx}]([\va],[\vx])$, $[A]=[\frac{d\vf}{d\va}]([\va],\tilde\vx)$, $C=(\imid[X])^{-1}$, $\tilde\vx=\imid[\vx]$ and $\tilde\va=\imid[\va]$. The operator \eqref{eq:pp-HS} will be denoted by $[H]_{\vf,[\va]}([\vx],[\vz])$ in the sequel. Experiments presented in the sequel will be carried out using the natural interval extensions of $\vf$, $\frac{d\vf}{d\va}$ and $\frac{d\vf}{d\vx}$.

\begin{remark}\label{rem:improvement}
	In \eqref{eq:pp-HS}, the expression $C\cdot[\vf](\tilde\va,\tilde\vx) + \bigl(C\cdot[A]\bigr)\cdot([\va]-\tilde\va)$ is used instead of $C\cdot\bigl([\vf](\tilde\va,\tilde\vx) + [A]\cdot([\va]-\tilde\va)\bigr)$ because it is a more efficient interval evaluation.
\end{remark}

The preconditioned parametric Hansen-Sengupta operator is compared to the preconditioned parametric Krawczyk operator
\begin{eqnarray}
	\label{eq:pp-K} [K]_{\vf,[\va]}([\vx])& := & \tilde\vx+[K]\Bigl(C\cdot[X] , -[\vy] , [\vx]-\tilde\vx \Bigr)
	\\ {[K]}([A],[\vb],[\vx]) & := & [\vb]+(I-[A])\cdot[\vx].
\end{eqnarray}
proposed in \cite{Rump1990}, where the same interval enclosure $[\vy]:=C\cdot[\vf](\tilde\va,\tilde\vx) + \bigl(C\cdot[A]\bigr)\cdot([\va]-\tilde\va)$ of $\{C\cdot\vf(\va,\tilde\vx)|\va\in[\va]\}$ is used in \eqref{eq:pp-K} and in \eqref{eq:pp-HS} (this point is not detailed in \cite{Rump1990}). As in \eqref{eq:pp-HS}, $[X]=[\frac{d\vf}{d\vx}]([\va],[\vx])$, $[A]=[\frac{d\vf}{d\va}]([\va],\tilde\vx)$, $C=(\imid[X])^{-1}$, $\tilde\vx=\imid[\vx]$ and $\tilde\va=\imid[\va]$. Theorem \ref{thm:natural-parametrization}, the inclusion $[K]_{\vf,[\va]}([\vx],[\vz])$ implies the existence of an unique solution to each system $\vf(\va,\cdot)=\vnull$ for $\va\in[\va]$.

\begin{example}\label{ex:academic-1}
	Let
	\begin{equation}
		\vf(\va,\vx):=
		\begin{pmatrix}
			(x_1 + a_1)^2 + (x_2 - a_2)^2 - 1\\ (x_1 - a_1)^2 + (x_2 + a_2)^2 - a_3^2
		\end{pmatrix}
		,
	\end{equation}
	$[\va]:=([\frac{1}{2}-\epsilon,\frac{1}{2}+\epsilon],[-\epsilon,\epsilon],[1-\epsilon,1+\epsilon])$ and $[\vx]:=([-0.2,0.2],[-0.7,1.1])$. The uncertainties width is set to $\epsilon=0.025$. The set $\{ \ \vx\in[\vx] \ | \ \exists \va\in[\va],\vf(\va,\vx)=\vnull \ \}$ is approximately represented on the left hand side graphic of Figure \ref{fig:Example1} solving the $2\times 2$ system of equations for a finite set of parameters values inside $[\va]$.
	
	Both operators \eqref{eq:pp-HS} and \eqref{eq:pp-K} are used to improve the initial enclosure $[\vx]$ by computing the sequences
	\begin{eqnarray}
		[\vx_{k+1}] & = & [H]_{\vf,[\va]}([\vx_k],[\vx_k])
		\\ {[\vy_{k+1}]} & = & [K]_{\vf,[\va]}([\vy_k])
	\end{eqnarray}
	with $[\vx_0]=[\vy_0]=[\vx]$. The following table summarizes the results obtained with both operators:

	\begin{center}
	\begin{tabular}{|c|c|c|}
		\hline & Final enclosure & Existence step
		\\\hline Hansen-Sengupta & $([-0.074, 0.075],[0.831, 0.901])$ & 3
		\\\hline  Krawczyk & $([-0.074, 0.075],[0.831, 0.901])$ & 2
		\\\hline 
	\end{tabular}
	\end{center}
	
	The final enclosure is also shown on the left hand side graphic of Figure \ref{fig:Example1}. These results seem to show that the parametric Krawczyk operator is sharper than the parametric Hansen-Sengupta operator: they both compute the same final enclosure while the former proves the existence one step before. This is surprising since in their non parametric form the Hansen-Sengupta operator is proved to be sharper in general than the Krawczyk operator (cf. \cite{Neum90}). However, a closer study shows that the Hansen-Sengupta operator is actually sharper: The right hand side graphic of Figure \ref{fig:Example1} shows the ratio
	\begin{equation}
		\frac{||\iwid([\vy_k])||-\iwid([\vx_k])||}{\iwid([\vx_k])||}.
	\end{equation}
	As this graphic shows, the enclosure computed by the Hansen-Sengupta operator is alway sharper. The difference is sensible at the first iterations (reaching approximately $20\%$ at step 5), and converges to $0$ as $k$ goes to infinity (the dashed line corresponds to $12\exp(-0.46 k)$ for information about the convergence rate to $0$). This also explains why the existence proof is posponed of one step for the Hansen-Sengupta operator: the enclosure computed at step $k=1$ by this latter operator is too sharp to obtain $[\vx_{2}] \subseteq \mathrm{int}\,[\vx_1]$.

\end{example}

\begin{figure}[t!]
	\centering
	\includegraphics[width=.8\textwidth,draft=\draft]{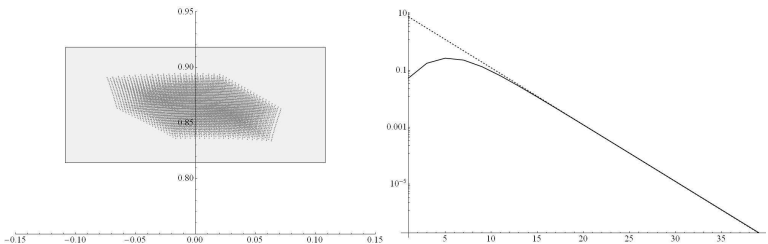}
	\caption{\label{fig:Example1}Left: Approximate solution set and its bounding box computed in Example \ref{ex:academic-1}. Right: Ratio of the enclosures widths for the parametric Hansen-Sengupta and Krawczyk operators plotted for each step.}
\end{figure}


\section{Rigorous Sensitivity Analysis}\label{s:application}

A direct application of Theorem \ref{thm:natural-parametrization} requires an initial domain. However, in practice this initial domain is often not available. Instead, an approximate solution $\vx^*$ for a nominal parameter value $\va^*\in[\va]$ is available. In the sequel, $\vx^*$ is supposed to satisfy exactly $\vf(\va^*,\vx^*)=\vnull$, but the usage of an approximate solution has no incidence in practice. From the sensitivity analysis point of view, we need to prove that each parameter $\va\in[\va]$ is mapped to an unique solution $\vx$ and to enclose the set of these solutions.

Provided that the parameters domains are small enough, the iteration
\begin{equation}
	[\vx_{k+1}]=[H]([\vx_k],\R^n) \ ; \ [\vx_0]=\vx^*
\end{equation}
will converge to $[\vx_\infty]$ which satisfies $[\vx_{\infty}]=[H]([\vx_\infty],\R^n)$. Although this limit can be proved to contain the solution set, the inclusion $[\vx_{k+1}]\subseteq \mathrm{int}[\vx_k]$ will never be satisfies because this iteration somehow translates and inflates the initial approximation $\vx^*$. It is more practical to additionally inflate each iterate of a fixed ration $\delta$ so as to obtain the inclusion $[\vx_{k+1}]\subseteq \mathrm{int}[\vx_k]$ after a finite number of steps, and hence prove the inclusion of the solution set. This leads to Algorithm \ref{alg:inflate}.

\incmargin{1em}
\linesnumbered
\begin{algorithm}
	\KwIn{$\vf:\R^p\times\R^n\rightarrow\R^n$, $[\va]\in\IR^p$, $\vx^*\in\R^n$}
	\KwOut{$[\vx]\in\IR^n$}
	$k_\mathrm{max}\leftarrow 10$;~\tcc{Maximum number of iterations}
	$\delta\leftarrow 1.01$;~\tcc{Inter-step inflation ratio}
	$\mathrm{success}\leftarrow\mathrm{\bf false}$\;
	$[\vx]\leftarrow \vx^*$\;
	\Repeat{$(~k>k_\mathrm{max}~)$}
	{
		$[\vx']\leftarrow \imid[\vx]+\delta \ ([\vx]-\imid[\vx])$\label{line:inflation}\;
		$[\vx]\leftarrow [H]_{\vf,[\va]}([\vx'],\R^n)$\;
		\lIf{$(~[\vx]\subseteq\mathrm{int}[\vx']~)$}{$\mathrm{success}\leftarrow\mathrm{\bf true}$\;}
	}
	\lIf{$(~\neg\,\mathrm{success}~)$}{$[\vx]\leftarrow\R^n$\;}
	\Return{$(~[\vx]~)$}\;
	\caption{\label{alg:inflate}}
\end{algorithm}

\begin{remark}
	Fixing a maximum number of steps $k_\mathrm{max}$ forces the termination of the algorithm. Smarter stopping criteria can easily be implemented. Also, it can be noted that once the existence proof has succeeded, the iteration becomes contracting and encloses the solution set. Therefore, Line \ref{line:inflation} can be replaced by $[\vx']\leftarrow [\vx]$ once $\mathrm{success}$ is $\mathrm{\bf true}$.
\end{remark}

As shown by the next examples, Algorithm \ref{alg:inflate} can be used as a rigorous sensibility analysis of the solution to a parametric system of equation: Being simply given an approximate solution for a nominal parameter value, Algorithm \ref{alg:inflate} allows rigorously bounding the variations of the solutions w.r.t. the variations of parameters.

\begin{example}\label{ex:academic-2}
	Let $\vf$ and $[\va]$ be defined as in Example \ref{ex:academic-1}, and consider the approximate solution $\vx^*:=(0.01,0.85)$, which is represented by a cross in Figure \ref{fig:InitialApproximateSolution}. The first steps of Algorithm \ref{alg:inflate} are represented by dashed boxes on Figure \ref{fig:InitialApproximateSolution}. The existence is proved at after four iterations, and hence the solution set is enclosed.
\end{example}

\begin{figure}[t!]
	\centering
	\includegraphics[width=0.5\textwidth,draft=\draft]{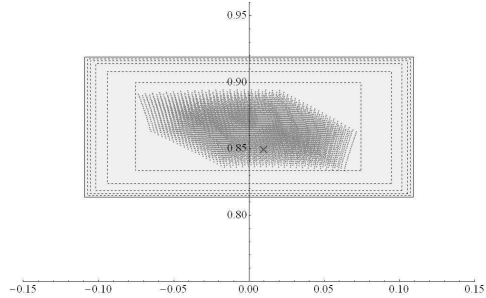}
	\caption{\label{fig:InitialApproximateSolution}Approximate solution set and its bounding box computed in Example \ref{ex:academic-2}. The initial approximate solution is represented by a cross and the boxes computed by the first steps in dashed lines.}
\end{figure}

\paragraph{\bf Related Work} In addition to \cite{Rump1990}, we have found two references \cite{Neumaier1989,Wallner2005} which use interval analysis for rigorous sensitivity analysis. The advantage of the method presented here is that it does not require an initial enclosure of the variations.




\end{document}